\documentclass[10pt,a4paper,reqno]{amsart}

\usepackage{amssymb}
\usepackage{enumitem}
\usepackage[ruled,noline,titlenumbered,linesnumbered,noend,norelsize]{algorithm2e}
\DontPrintSemicolon
\SetNlSty{footnotesize}{}{}
\IncMargin{0.5em}
\SetNlSkip{1em}
\SetKwIF{If}{ElseIf}{Else}{if}{then}{else if}{else}{endif}
\SetKw{Return}{return}
\SetKw{Pick}{pick}
\SetNlSkip{2em}
\newenvironment{algorithm-hbox}{\hbadness=10000\begin{algorithm}}{\end{algorithm}}

\theoremstyle{plain}
\newtheorem{theorem}{Theorem}

\newtheorem{guess}[theorem]{Conjecture}
\newtheorem{lemma}[theorem]{Lemma}
\newtheorem{corollary}[theorem]{Corollary}

\newtheorem*{observation}{Observation}

\newtheorem*{definition}{Definition}

\numberwithin{equation}{section}

\newenvironment{enumeratei-continued}{\begin{enumerate}[label=\textup{(\roman*)}, noitemsep, topsep=1.5mm, labelindent=.8em, leftmargin=*, widest=.]}{\end{enumerate}}

\renewenvironment{itemize}{\begin{enumerate}[label=$*$, noitemsep, topsep=1.5mm, labelindent=.8em, leftmargin=*, widest=.]}{\end{enumerate}}

\newcommand{\set}[1]{\{#1\}}
\newcommand{\norm}[1]{{|#1|}}

\renewcommand{\epsilon}{\varepsilon}

\newcommand{\case}[2]{\noindent \textbf{Case #1.} #2}
\newcommand{\claim}[2]{\noindent \textbf{Claim #1.} #2}

\renewcommand{\leq}{\leqslant}
\renewcommand{\geq}{\geqslant}
\renewcommand{\le}{\leqslant}
\renewcommand{\ge}{\geqslant}

\newcommand{\NN}{\mathbb{N}}

\newcommand{\olch}{{\rm  ch}^{ \rm OL}}

\DeclareMathOperator{\ch}{ch}
\DeclareMathOperator{\choice}{ch^{OL}}
\DeclareMathOperator{\chnol}{ch}
\DeclareMathOperator{\col}{col}

\begin{document}
\title[Towards on-line Ohba's conjecture]{Towards on-line Ohba's conjecture}
\author{Jakub Kozik}
\address{Theoretical Computer Science Department, Faculty of Mathematics and
Computer Science, Jagiellonian University, Krak\'{o}w, Poland}
\email{jkozik@tcs.uj.edu.pl}
\author{Piotr Micek}
\address{Theoretical Computer Science Department, Faculty of Mathematics and
Computer Science, Jagiellonian University, Krak\'{o}w, Poland}
\email{piotr.micek@tcs.uj.edu.pl}
\author{Xuding Zhu}
\address{Department of Mathematics, Zhejiang Normal University, China, and
Department of Applied Mathematics, National Sun Yat-sen University, Taiwan}
\email{xudingzhu@gmail.com}
\date{\today}
\thanks{
This research is supported by: NSFC No. 11171730 and  ZJNSF No. Z6110786, and Polish National Science Center UMO-2011/03/D/ST6/01370.}

\begin{abstract}
The on-line choice number of a graph is a variation of the choice number  defined through a two person game.
It is at least as large as the choice number for all graphs and is strictly larger for some graphs. In particular, there are
graphs $G$ with $|V(G)| = 2 \chi(G)+1$ whose on-line choice numbers are larger than their chromatic numbers,
in contrast to a recently confirmed conjecture of Ohba that
every graph $G$ with $|V(G)| \le 2 \chi(G)+1$
has its choice number equal its chromatic number.
Nevertheless, an on-line version of Ohba conjecture was proposed   in
[P. Huang, T.  Wong and X. Zhu,  Application of polynomial method to on-line colouring of graphs, European J. Combin., 2011]:
Every graph $G$ with $|V(G)| \le 2 \chi(G)$ has its  on-line choice number equal its chromatic number.
This paper confirms the on-line version of Ohba conjecture  for graphs $G$ with independence number at most $3$.
We also study list colouring of  complete multipartite graphs $K_{3\star k}$ with all parts of size $3$.
We prove that the on-line choice number of $K_{3 \star  k}$ is at most $\frac{3}{2}k$,
and present an alternate proof of Kierstead's result that its choice number is $\lceil (4k-1)/3 \rceil$.
For general graphs $G$, we prove that if $|V(G)| \le \chi(G)+\sqrt{\chi(G)}$ then its  on-line choice number  equals chromatic number.
\end{abstract}

\maketitle

\section{Introduction}

\noindent A {\em list assignment } of a graph $G$ is a mapping $L$ which
assigns to each vertex $v$ a set $L(v)$ of permissible colours. An {\em
$L$-colouring} of $G$ is a proper vertex colouring   of $G$ which colours each vertex with
one of its permissible colours. We say that $G$ is {\em
$L$-colourable} if there exists an $L$-colouring of $G$. A graph $G$
is called $k$-choosable if for any list assignment  $L$ with
$|L(v)|=k$ for all $v \in V(G)$, $G$ is $L$-colourable. More
generally, for a function  $f: V(G) \to \NN$, we say $G$ is   {\em
$f$-choosable} if for every list assignment  $L$ with $|L(v)|=f(v)$ for all $v \in V(G)$,
$G$ is $L$-colourable.   The {\em choice number} $\ch(G)$ of $G$ is
the minimum $k$ for which $G$ is $k$-choosable.   List colouring of
graphs has been studied extensively in the literature  (cf. \cite{Vizing76,ERT79,tuzasurvey}).

A list assignment  of a graph $G$ can be given alternatively as follows:
Without loss of generality, we may assume that $\cup_{v \in V(G)}
L(v) = \{1,2,\ldots, q\}$ for some integer $q$. For $i=1,2,\ldots, q$, let $V_i = \{v: i
\in L(v)\}$. The sequence $(V_1,V_2, \ldots, V_q)$ is   another
way of specifying the list assignment.   An $L$-colouring of
$G$ is equivalent to a sequence $(X_1, X_2, \ldots, X_q)$ of
independent sets that form a partition of $V(G)$ and such that $X_i
\subseteq V_i$ for $i=1,2,\ldots, q$. This point of view of list colouring motivates the definition of the following
list colouring game on a graph $G$, which was introduced in \cite{Schauz4,Schauz}.

\begin{definition}
Given a finite graph $G$ and a mapping $f: V(G) \to \mathbb{N}$, two players, the Lister and the Painter,
 play the following game.
In the $i$-th step, the Lister presents a non-empty subset $V_i$ of $V(G) \setminus \cup_{j=1}^{i-1}X_j$,
and  the Painter chooses an independent set $X_i$ contained in $V_i$.
If $v \in V_i$, then we say colour $i$ is a permissible colour of vertex $v$.
If $v \in X_i$, then we say $v$ is coloured by colour $i$.
If for some integer
$m$, there is a vertex $v$  which  is uncoloured at the end of the $m$th step and is contained in
$f(v)$ of the $V_j$'s with $j \le m$,  then the game ends and the Lister wins the game. Otherwise, at some
step, all vertices are coloured, the game ends and  the Painter  wins the game.
\end{definition}

We call such a game the {\em on-line $(G,f)$-list colouring game}.
 We say $G$ is {\em on-line $f$-choosable} if the Painter
has a winning strategy in the on-line $(G,f)$-list   colouring game, and we say
   $G$ is  {\em on-line $k$-choosable}
if $G$ is on-line $f$-choosable for the constant function $f
\equiv k$. The {\em on-line choice number} of $G$, denoted by $\olch(G)$, is the
minimum $k$ for which $G$ is on-line $k$-choosable.


It follows from the definition that
for any graph $G$, $\olch(G) \ge \ch(G)$.
There are graphs $G$ with $\olch(G) >  \ch(G)$ (see \cite{Zhuonline}).
However, many currently known upper bounds for the
choice numbers of  classes of graphs  remain upper bounds for their on-line choice
numbers. For example, the on-line choice number of planar graphs is
at most $5$ \cite{Schauz}, the on-line choice number of  planar graphs of girth at least $5$
is at most $3$ \cite{Schauz,CZ2011},
the on-line choice number of the line graph $L(G)$ of a
bipartite graph $G$ is  $\Delta(G)$ \cite{Schauz}, and if $G$ has an orientation in
which the number of even eulerian subgraphs differs from the number
of odd eulerian subgraphs and $f(x) = d^+(x)+1$, then $G$ is on-line
$f$-choosable \cite{Schauz4}.

To prove an upper bound for $\olch(G)$, one needs to design an on-line algorithm that produces a proper list colouring,
while to prove an upper bound for $\ch(G)$, it suffices to  prove the existence of a proper list colouring.
For an upper bound $k$ for $\ch(G)$,  a proof of $k$ being an upper bound for $\olch(G)$
usually leads to better
understanding of the original result and problem.
For example,
Schauz \cite{Schauz4}
generalized the Alon-Tarsi upper bounds for choice number to on-line choice number. This work
leads to a pure combinatorial proof of the Alon-Tarsi upper bounds which was originally proved by
using polynomial method. Alon proved that the fractional choice number of a graph always equals its fractional chromatic number.
Gutowski \cite{Gutowski2011} generalized this result  and proved that the on-line fractional choice number
of a graph always equals its fractional chromatic number. Gutowski's proof, which is a simple algorithm,
is conceptually simpler than the original proof which uses probabilistic method.
Another result of Alon says the choice number of any graph $G$ is bounded above by
$C \chi(G) \log_2 |V(G)|$ for some constant $C$. Zhu \cite{Zhuonline} showed that upper bound is also an upper bound for the on-line choice number.
The original proof uses probabilistic argument, and a simple algorithm is given in \cite{Zhuonline} to produce the required list colouring.

A graph $G$ is called {\em chromatic-choosable}   if
$\chi(G) =  \ch(G)$. The problem as which graphs are chromatic-choosable
has been extensively studied. A few well-known classes of graphs are conjectured to be chromatic-choosable. These include
line graphs  (conjectured
independently by Vizing, by Gupta, by Albertson and Collins, and by
Bollob\'{a}s and Harris, see \cite{HC1992} and \cite{JensenToft}), claw-free graphs (conjectured by Gravier and Maffray \cite{GM1997}),
 square of graphs (conjectured by Kostochka and Woodall \cite{KW2001}), graphs $G$ with  $|V(G)| \le 2\chi(G)+1$ (conjectured by Ohba \cite{OH2002}).
 All these conjectures remain open, except that the Ohba Conjecture has recently been confirmed by
  Noel. Reed and Wu \cite{NRW12}:

\begin{theorem}
  \label{thmNRW} If $|V(G)| \le 2\chi(G)+1$, then ${\rm ch}(G)=\chi(G)$.
\end{theorem}

We call a graph $G$ {\em on-line chromatic-choosable} if  $\chi(G)=\olch(G)$.
It is natural to ask if the classes of graphs mentioned above are also on-line chromatic choosable.
It is proved in \cite{KKLZ2011} that the complete multipartite graph
$G=K_{3, 2\star k}$ with $k$ parts of size $2$ and $1$ part of size $3$ has $\olch(G)=k+2$.
So, in contrast to Theorem \ref{thmNRW},  there are graphs $G$ with $|V(G)| = 2\chi(G)+1$
and yet  $\olch (G)> \chi(G)$.
Nevertheless, experiments and some earlier results
show that a slightly modified version of Ohba's conjecture might be true in the on-line case.
The following conjecture was proposed in \cite{HWZ2010}.

\begin{guess}
\label{g4} If $|V(G)| \le 2\chi(G)$, then  $\chi(G)=\olch(G)$.
\end{guess}
\noindent

The on-line version of Ohba conjecture is largely open, although some special cases are verified \cite{HWZ2010,KKLZ2011}.
Some of the tools that are crucial to the study  of Ohba's conjecture do not apply to the on-line version.
In all the papers that study Ohba's Conjecture, including the paper by Noel, Reed and Wu \cite{NRW12} which finally
proves this conjecture, an important step is an application of Hall Theorem concerning matching in bipartite graphs.
However, Hall Theorem does not apply to the on-line version.
New techniques are needed for the study of on-line version of Ohba's conjecture.

The main focus of this paper is the on-line version of Ohba's conjecture.
We confirm Conjecture \ref{g4} for graphs with independence number at most $3$.
The proof is by induction. The most difficult part of the work is
not the proof itself, but to nail down a proper technical
statement, i.e., Lemma \ref{k3k:induction}, which can be proved by
induction and implies Conjecture \ref{g4} for graphs with independence number at most $3$.
 Another  consequence of   Lemma \ref{k3k:induction} is that
  $\olch(K_{3\star k}) \le \frac32k$, where $K_{3 \star k}$ is the complete $k$-partite
graph with each part of size $3$. The choice number of these graphs were studied
by Erd\"{o}s, Rubin and Taylor \cite{ERT79}, and Kierstead \cite{Kierstead00}.
It was proved in \cite{ERT79} that   $\ch(K_{3\star k}) \ge \lceil (4k-1)/3 \rceil$, and proved in
  \cite{Kierstead00}  that this lower bound is tight: for any positive integer $k$,
$\ch(K_{3\star k}) = \lceil (4k-1)/3 \rceil$.
By proving a technical result similar to Lemma \ref{k3k:induction},
we give an alternate proof of Kierstead's result.

As mentioned above,
we know that there are graphs $G$ for which $\olch(G) > \ch(G)$.
However, it remains a challenging open problem to determine whether the difference $\olch(G)-\ch(G)$ can be arbitrarily large.
As the exact values of $\ch(K_{3\star k})$ are known (there are not many graphs whose choice numbers are known), the  graphs $K_{3 \star k}$
are natural  candidates for proving a hypothetic separation (by more than a constant) of choice number and on-line choice number.
It would be interesting to determine the exact value of $\olch(K_{3\star k})$ or prove lower bounds for $\olch(K_{3\star k})$
that separates their choice number and on-line choice number by more than a constant.


\section{The join of $G$ and $K_n$}

\noindent This section proves that for any graph $G$,  by adding enough universal vertices, one can construct a graph that is on-line chromatic-choosable. For two graphs $G$ and $G'$, the {\em join} of $G$ and $G'$, denoted by $G  +  G'$ is the graph obtained from the disjoint union of
$G$ and $G'$ by adding all the possible edges between $V(G)$ and $V(G')$.

\begin{theorem}\label{thm:existence}
For every graph $G$ there exists a positive integer $n$ such that $\chi(G  +  K_n)=\choice(G +  K_n)$.
\end{theorem}
\begin{proof}
Without loss of generality, we may assume that $G$ is a complete $\chi(G)$-partite graph.
Let us start with an easy observation (see \cite{Schauz}):   Assume $H$ is a graph and $f:V(H)\to \NN$ is a function. If $f(v)\geq d_H(v)+1$ for a vertex $v\in V(H)$, then
\begin{center}
$H$ is on-line $f$-choosable if and only if $H-v$ is on-line $f$-choosable.
\end{center}

For a given graph $G$, we put $H_0=G +  K_n$ with $n=\norm{V(G)}^2$ and $f(v)=\chi(H_0)=\chi(G)+n$ for all $v\in V(H_0)$. Let $V_1,\ldots, V_{\chi(H_0)}$ be a partition of $V(H_0)$ into independent sets. We are going to present a winning strategy for the Painter  in the on-line $(H_0,f)$-list colouring game.

We denote by $H_i$ a subgraph of all uncoloured vertices of $H_0$ after $i$ steps, and $f_i(v)$ be the number
of remaining permissible colours for $v$ after $i$ step.
Before playing the $(i+1)$-th step, we delete from $H_i$, one by one,
all the vertices $v$ with $f_i(v)\geq d_{H_i}(v)+1$ (by using the observation above).
The resulting graph is still denoted by
$H_i$. Now, by a \emph{part} of $H_i$ we mean a non-empty set of the form $V_j\cap H_i$ for $1\leq j\leq \chi(H_0)$.
Assume at the $(i+1)$th step, the Lister chooses a subset $U_i$. the Painter   finds an independent set $I$ contained in $U_i$ according to the
following algorithm.

\begin{algorithm}
\caption{Strategy for the Painter in the $(i+1)$-th step}
\uIf{\textup{there is a part $V$ of $H_i$ with $\norm{V}\geq2$ and $V\subseteq U_i$}\label{alg:large-part-if}}
{
\Pick $I=V$\;\label{alg:large-part-taken}
}
\uElseIf{\textup{there is a part $V$ of $H_i$ with $\norm{V}=1$ and $V\subseteq U_i$}}
{
\Pick $I=V$\;\label{alg:small-part-taken}
}
\uElse
{
\Pick $I$ to be any maximal independent set in $U_i$\;\label{alg:weak-move}
}
\end{algorithm}

For $v\in V(G\cap H_i)$, 
define the \emph{deficit} of $v$ as $d_{H_i}(v)+1-f_i(v)$,
which is the number of additional permissible colours needed so that $v$ can be removed from the graph
(by the observation we started with). Since vertices $v$ with $f_i(v) \ge d_{H_i}(v)+1$
are removed, we know that the deficit of each vertex $v$ is positive.
The deficit of a part $V$ of $H_i$ is the sum of deficits of its vertices
\[
	\sum_{v\in V}(d_{H_i}(v)+1-f_i(v)).
\]

We will show that after every step of the game the deficit of each part of size at least $2$ decreases.
Let $V$ be a part of $H_i$ and $\norm{V}\geq 2$.

If line \ref{alg:large-part-taken} is executed, then
either part $V$ is picked and it disappears in $H_{i+1}$, or $d_{H_{i+1}}(v)\leq d_{H_i}(v)-2$ and $f_{i+1}(v)\geq f_i(v)-1$ for
all $v \in V$. Hence the deficit of each vertex of $V$ decreases.

If line \ref{alg:small-part-taken} is executed,
then $d_{H_{i+1}}(v)=d_{H_i}(v)-1$, $f_{i+1}(v)\geq f_i(v)-1$ for all $v\in V$, and there exists $v\in V$
with $f_{i+1}(v)=f_i(v)$ as $V$ is not contained in $U_{i+1}$.
So the total deficit of $V$ decreases.

Assume line \ref{alg:weak-move} is executed. If $I=V\cap U_{i+1}$,
then $d_{H_{i+1}}(v)=d_{H_i}(v)$, $f_{i+1}(v)=f_i(v)$ for all $v\in V- U_{i+1}$ so
the sum decreases as the deficit of erased vertices is positive.
Otherwise,
$d_{H_{i+1}}(v)\leq d_{H_i}(v)-1$ and $f_{i+1}(v) \ge f_i(v)-1$
for all $v \in V$ and there exists $v \in V$ with $f_{i+1}(v)=f_i(v)$.
So the deficit of $V$ decreases.

As each $v\in V(H_0)$ has deficit bounded by $\norm{V(G)}$,   each part has
initial deficit bounded by $n=\norm{V(G)}^2$.  Since after each step the deficit of each part
of size at least $2$ decreases and vertices with non-positive deficit are deleted,
after $n$ rounds the remaining graph, namely $H_n$, forms a clique.

The vertices in $H_n$ may come from $G$ or $K_n$ and there are at most $\chi(G)$ vertices coming
from $G$, at most one for each part of $G$.  If $U_i\cap K_n \neq \emptyset$ then the number of
parts in $H_{i+1}$ decreases by $1$ comparing to the number of parts in $H_i$
(as line \ref{alg:large-part-taken} or \ref{alg:small-part-taken} is executed).
Therefore
\begin{align*}
f_n(v)&\geq \text{the number of parts in $H_n$}= d_{H_n}(v)+1&&\text{for all $v\in H_n \cap K_n$}
\end{align*}
For vertices $v \in H_n \cap G$, as each step decreases the number of permissible colours by at most $1$, we have $f_n(v) \ge f_0(v) - n = \chi(G)$.
By applying the observation repeatedly, these inequalities certify that all vertices of $H_q$ are removed and $H_q$
is empty, which finishes the proof.
\end{proof}

The argument presented gives also an Ohba-like statement with much more restricted constraint on the size and the chromatic number of
a graph.

\begin{corollary}
If $\norm{V(G)}\leq \chi(G)+\sqrt{\chi(G)}$, then $\chi(G)=\olch(G)$.
\end{corollary}

\section{A lemma and main result}
\noindent In the remainder of this paper, we consider complete multipartite graphs of independence
number at most $3$.  For   integers $k_1, k_2, k_3 \ge 0$, we denote by $K_{3\star k_3,2\star k_2,1\star k_1}$
the complete multipartite graph with $k_1$ parts of size $1$, $k_2$ parts of size $2$ and $k_3$ parts of size $3$.
Lemma \ref{k3k:induction} below specifies a sufficient condition for such a graph $G$
to be on-line $f$-choosable. We use it to derive results for on-line Ohba conjecture and on-line choosability of graphs with independence number 3.

For a subset $U$ of $V(G)$, let $\delta_U: V(G) \to \{0,1\}$ be the characteristic function of $U$, i.e., $\delta(x)=1$ if $x \in U$ and
$\delta_U(x)=0$ otherwise. The following observation follows directly from the definition of the on-line $(G,f)$-colouring game (see \cite{Schauz}).

\begin{observation}
If $G$ is an edgeless graph and $f(v)\geq 1$ for all $v\in V(G)$, then $G$ is on-line $f$-choosable. If $G$ has at least one edge, then $G$ is on-line $f$-choosable if and only if for every $U\subseteq V(G)$, there is an independent set $I$ of $G$ such that $I\subseteq U$ and $G-I$ is on-line $(f-\delta_U)$-choosable.
\end{observation}

\begin{lemma}
\label{k3k:induction}
	Let $G$ be a complete multipartite graph $G$ with each part of size at most $3$.
Let  $\mathcal{A}, \mathcal{B}, \mathcal{C}, \mathcal{S}$ be a partition of the set of
parts of $G$ into classes  such that $\mathcal{A}$ contains only parts of size $1$,
$\mathcal{B}$ contains only parts of size $2$, $\mathcal{C}$ contains only parts of size
$3$ and $\mathcal{S}$ contains parts of size $1$ or $2$. Let $k_1, k_2, k_3, s$ denote the
cardinalities of classes $\mathcal{A}, \mathcal{B}, \mathcal{C}, \mathcal{S}$, respectively. Suppose
that classes $\mathcal{A}$ and $\mathcal{S}$ are ordered i.e. $\mathcal{A}= (A_1, \ldots, A_{k_1})$
and $\mathcal{S}=(S_1, \ldots, S_s)$. For $1 \le i \le s$, let $v_S(i)= \sum_{1\leq j< i} \norm{S_i}+1$.
Assume $f:V(G)\to \NN$ is a function for
which the following conditions hold
\begin{alignat}{2}
f(v)&\geq k_3+k_2+i, &&\quad\text{for all $1\leq i \leq k_1$ and $v\in A_i$}\tag{1}\label{inv:1}\\
f(v)&\geq 2k_3+k_2+k_1+v_S(i), &&\quad\text{for all $1\leq i \leq s$ and $v\in S_i$}\tag{1'}\label{inv:1'}\\
f(v)&\geq k_3+k_2, &&\quad\text{for all $v\in B\in\mathcal{B}$}\tag{2.1}\label{inv:2.1}\\
\sum_{v\in B}f(v)&\geq \norm{V(G)}, &&\quad\text{for all $B\in\mathcal{B}$}\tag{2.2}\label{inv:2.2}\\
f(v)&\geq k_3+k_2, &&\quad\text{for all $v\in C\in\mathcal{C}$}\tag{3.1}\label{inv:3.1}\\
f(u)+f(v)&\geq \norm{V(G)}-1, &&\quad\text{for all $u,v\in C\in\mathcal{C}$, $u\neq v$}\tag{3.2}\label{inv:3.2}\\
\sum_{v\in C}f(v)&\geq \norm{V(G)}-1+k_3+k_2+k_1, &&\quad\text{for all $C\in\mathcal{C}$}.\tag{3.3}\label{inv:3.3}
\end{alignat}
Then $G$ is on-line $f$-choosable.
\end{lemma}
\begin{proof}
The proof goes by induction on $\norm{V(G)}$. If $G$ is edgeless, i.e., $k_1+k_2+k_3+s=1$, then $G$ is
on-line $f$-choosable as $f(v)\geq 1$ for all $v\in V(G)$. Assume now that $G$ has at least two parts
and that the statement is verified for all smaller graphs.

Given $U\subseteq V(G)$, we shall find an independent set $I$ of $G$ such that $I\subseteq U$ and $G-I$
is on-line $(f-\delta_U)$-choosable. Let $G'=G-I$ and $f'=f-\delta_U$.
Note that $f'(v)\geq f(v)-1$ for all $v\in V(G)$. Clearly, $G'$ is also a complete
multipartite graph with each part of size at most $3$. We are going to show that $G'$
with $f'$, an appropriate partition $\mathcal{A'}, \mathcal{B'}, \mathcal{C'}, \mathcal{S'}$
and orderings of $\mathcal{A'}$ and $\mathcal{S'}$  fulfill the conditions of
Lemma \ref{k3k:induction}. Hence, by induction hypothesis $G'$ is on-line $f'$-choosable.

The strategy of choosing an independent set $I$ is given by the case distinction.
Note that we consider the setting of Case $i$ only when the conditions for all $i-1$
previous cases do not hold. When we verify the inequalities from the statement of
Lemma \ref{k3k:induction} for $G'$ and $f'$ we usually compare the total decrease/increase
of left and right hand sides with the analogous inequalities that hold for $G$ and $f$.
The notation for the parts of $G'$ and its sizes is analogous as for $G$, e.g. $A_i'$, $S_i'$,
$k'_1$, $s'$ and so on. Partition $\mathcal{A'}, \mathcal{B'}, \mathcal{C'}, \mathcal{S'}$
and orders on the classes $\mathcal{A'}$ and $\mathcal{S'}$ are usually inherited.
In the case distinction below we comment the partitions only  if the order or partition changes in the considered step.

\smallskip

\case{1}{$C\subseteq U$ for some $C\in\mathcal{C}$.}

\smallskip

Put $I=C$. Then $k_3'=k_3-1$ and all other parameters remain the same. Note that $\norm{V(G')}=\norm{V(G)}-3$. Now, it is immediate that $G'$ with inherited partition and $f'$ satisfy the conditions of Lemma \ref{k3k:induction}.

\smallskip

\case{2}{$B\subseteq U$ for some $B\in\mathcal{B}$.}

\smallskip

Put $I=B$. Then $k_2'=k_2-1$ and all other parameters remain the same. Note that $\norm{V(G')}=\norm{V(G)}-2$. Again, it is immediate that $G'$ with inherited partition and $f'$ satisfy the conditions of Lemma \ref{k3k:induction}. Note that because Case 1 does not apply, for inequality (3.3), the left-hand side decreases by at most $2$.

\smallskip

In all remaining cases, as conditions for  cases 1 and 2 do not hold, we have
\begin{enumeratei-continued}
\item $U$ covers at most one vertex in each $B\in \mathcal{B}$ (we are not in Case 2). This implies that inequalities \eqref{inv:2.2} for any $G'$ will trivially hold provided $\norm{V(G')}\leq \norm{V(G)}-1$.
\item $U$ covers at most two vertices in each $C\in\mathcal{C}$ (we are not in Case 1).\label{enu:3.3}
\end{enumeratei-continued}

\smallskip

\case{3}{There is $C\in\mathcal{C}$ with $U\cap C=\set{u,v}$ and \eqref{inv:3.1} is saturated for $v$ or \eqref{inv:3.2} is saturated for $u$ and $v$.}

\smallskip

Let $C=\set{u,v,w}$. Put $I=\set{u,v}$. Then $k_3'=k_3-1$, $k_1'=k_1+1$ and all other parameters remain unchanged.
Indeed, we colour two vertices of $C$ and the remaining vertex forms $A'_{k_1'}=\set{w}$, a new part of size $1$,
which is appended to the ordering of $\mathcal{A'}$. Note that $\norm{V(G')}=\norm{V(G)}-2$.

Now, we need to check that all the inequalities of Lemma \ref{k3k:induction} hold for $G'$ and $f'$.
Inequality \eqref{inv:1} holds for $A'_i$ with $1\leq i < k_1'$ as the right hand side decreases by $1$
and the left hand side decreases at most by $1$. Inequality \eqref{inv:1} holds for $A'_{k_1'}=\set{w}$
either because  \eqref{inv:3.2} is saturated for $u,v$ in $G$ and hence
\[
f'(w)=f(w)\geq \norm{V(G)}-1+k_3+k_2+k_1 - (\norm{V(G)}-1)= k_3'+k_2'+k_1',
\]
or because \eqref{inv:3.1} is saturated for $v$ in $G$ and hence
\[
f'(w)=f(w)\geq \norm{V(G)}-1-k_3-k_2\geq 2k_3+k_2+k_1-1= 2(k_3'+1)+k_2'+(k_1'-1)-1.
\]
The inequality \eqref{inv:3.3} for $C\in\mathcal{C}$ holds as the right hand side decreased by $2$ and the left hand side decreased by at most $2$ (see \ref{enu:3.3}). The other inequalities hold trivially.

\smallskip

Note that in all remaining cases
\begin{enumeratei-continued}
\item For each $C\in\mathcal{C}$ either $\norm{U\cap C}\leq 1$, or $\norm{U\cap C}=2$ and \eqref{inv:3.2} is not saturated for $U\cap C$ in $G$ (we are not in Case 3). This implies that inequalities \eqref{inv:3.2} will hold for any $G'$ provided $\norm{V(G')}\leq \norm{V(G)}-1$.
\end{enumeratei-continued}

\smallskip

\case{4}{There is $B\in\mathcal{B}$ with $U\cap B=\set{v}$ and \eqref{inv:2.1} is saturated for $v$.}

\smallskip

Let $B=\set{u,v}$. Put $I=\set{v}$. Then $k_2'=k_2-1$ and $s'=s+1$ and all other parameters remain unchanged. The part  $\set{u}$ form a new part of size $1$ and is appended at the end of the order to the class $\mathcal{S}$ as $S'_{s'}$. Note that $\norm{V(G')}=\norm{V(G)}-1$.

We are going to check the inequalities for $G'$ and $f'$. Inequalities \eqref{inv:1} for $A_j'$
with $1 \leq j \leq k_1'$ and \eqref{inv:1'} for $S_j^{'}$ with $1\leq j\leq s'-1$ hold as the
right hand side decreases by $1$ while the left hand side decreases at most by $1$.
Inequality \eqref{inv:1'} for $S'_{s'}=\set{u}$ holds by \eqref{inv:2.2} for $u,v$ in $G$
and the saturation of \eqref{inv:2.1} for $v$ in $G$
\[
f'(u)=f(u)\geq \norm{V(G)}-k_3-k_2 
=2k_3'+(k_2'+1)+k_1+(v_{S'}(s')-1).
\]
The inequalities \eqref{inv:2.1}, \eqref{inv:3.1} and \eqref{inv:3.3} for $G'$ with $f'$ hold trivially.

\smallskip

Note that in all remaining cases
\begin{enumeratei-continued}
\item   For all $v\in U\cap \bigcup_{B\in\mathcal{B}} B$ the inequality \eqref{inv:2.1} is not saturated for $v$ in $G$. This means that \eqref{inv:2.1} will hold in any $G'$.
\end{enumeratei-continued}

\smallskip
\case{5}{There is $C\in\mathcal{C}$ with $U\cap C=\set{v}$ and \eqref{inv:3.1} is saturated for $v$.}

\smallskip

Let $C=\set{u,v,w}$ and put $I=\set{v}$. The remaining part $\{u,w\}$ is
 appended at the end of the sequence $\mathcal{S}$. Note that $\norm{V(G')}=\norm{V(G)}-1$.

The inequalities \eqref{inv:1} for $A_j'$ with $1\leq j\leq k_1'$ and \eqref{inv:1'} for $S'_j$ with $1\leq j\leq s-1$ hold as the right hand side decreases by $1$ while the left hand side decreases at most by $1$. Inequalities \eqref{inv:1'} for the vertices of the new part $S'_{s'}=\set{u,w}$ hold
 because   \eqref{inv:3.1} is saturated for $v$ in $G$ and hence for $x \in \{u,w\}$,
\[
f'(x)=f(x)\geq (\norm{V(G)}-1)-k_3-k_2=2 (k_3'+1) + k_2'+ k_1'+ (v_{S'}(s')-1)-1,
\]
The inequalities \eqref{inv:2.1}, \eqref{inv:3.1} for $G'$ are trivial. The inequalities \eqref{inv:3.3} for $G'$ hold as the right hand side decreases by $2$ and the left hand side at most by $2$ (see \ref{enu:3.3}).

\smallskip

Note that in all remaining cases
\begin{enumeratei-continued}
\item For all $v\in U\cap \bigcup_{C\in\mathcal{C}} C$ the inequality \eqref{inv:3.1} is not saturated for $v$ in $G$. This means that \eqref{inv:3.1} will hold in any $G'$.
\end{enumeratei-continued}

\smallskip

\case{6}{There is $1\leq i \leq k_1$ with $A_i\subseteq U$.}

\smallskip

Let $i$ be the least index with $A_i=\set{v}\subseteq U$. Put $I=\set{v}$. Then $k_1'=k_1-1$ and all other parameters remain unchanged. We also renumber the parts of size $1$, namely $A_j'=A_{j+1}$ for $i\leq j \leq k_1'$. Note that $\norm{V(G')}=\norm{V(G)}-1$.

The inequality \eqref{inv:1} for $A_j'$ with $1\leq j <i$ holds as  both sides are the same   in $G'$ as in $G$. The inequality \eqref{inv:1} for $A_j'=A_{j+1}=\set{u}$ with $i\leq j \leq k_1'$ holds as
\[
f'(u)\geq f(u)-1\geq k_3+k_2+(j+1)-1.
\]
The inequalities \eqref{inv:3.3} hold in $G'$ as the right hand side decreased by $2$ and the left hand side decreased by at most $2$ (see \ref{enu:3.3}).

\smallskip
\case{7}{There is $C\in\mathcal{C}$ with $U\cap C=\set{u,v}$.}

\smallskip

Let $C=\set{u,v,w}$. Put $I=\set{u,v}$. Then $k_3'=k_3-1$ and $k_1'=k_1+1$ and all other parameters remain unchanged. There is one new part of size $1$, namely $A_1'=\set{w}$, and all the others are renumbered $A_{j}'=A_{j-1}$ for $2\leq j \leq k_1'$. Note that $\norm{V(G')}=\norm{V(G)}-1$.

The inequality \eqref{inv:1} for $A_1'=\set{w}$ holds by \eqref{inv:3.1} for $w$ in $G$
\[
f'(w)=f(w)\geq k_3+k_2=(k_3'+1)+k'_2.
\]
The inequality \eqref{inv:1} for $A_j'=A_{j-1}=\set{x}$ with $2\leq j\leq k_1'$ holds as $x\not\in U$ (Case 6 does not apply)
\[
f'(x)=f(x)\geq k_3+k_2+(j-1)=k_3'+k'_2+j.
\]
The inequalities \eqref{inv:3.3} hold in $G'$ as the right hand side decreased by $2$ and the left hand side decreased by at most $2$.

\smallskip

Note that in all remaining cases
\begin{enumeratei-continued}
\item $\norm{U\cap C}\leq 1$, for $C\in\mathcal{C}$.  As we always have $\norm{V(G')}\leq \norm{V(G)}-1$ and $k_3'+k_2'+k_1'\leq k_3+k_2+k_1$ the inequalities \eqref{inv:3.3} will hold for any $G'$.
\end{enumeratei-continued}

\smallskip

\case{8}{There is $1\leq i \leq s$ with $S_i \cap U \neq \emptyset $.}

\smallskip

Let $i$ be the least $i$ with $S_i \cap U \neq \emptyset$. Put $I=S_i \cap U$. Then $s'=s-1$ and all other parameters remain unchanged. If $|S_i \cap U|=S_i$,  we update the order of the parts in the sequence $\mathcal{S}$, in the following way, for $i\leq j \leq s'$ we put $S'_j=S_{j+1}$. If $|S_i \cap U|\ne S_i$ the order remains the same. Note that $\norm{V(G')}\le \norm{V(G)}-1$.

The inequalities \eqref{inv:1} for $A_j'$ with $1\leq j \leq k_1'$ and \eqref{inv:1'} for $S_j^{'}$ with $1\leq j <i$ hold as   both sides does not change.
For every vertex from parts $S_{j+1}, \ldots S_{s}$ the right hand side of the inequality \eqref{inv:1'} decreases by at least one, therefore inequalities hold. For the vertices from  $S_i\setminus U$ (this set may be empty) both sides of inequality does not change, therefore inequality holds as before.

\smallskip

Note that in all remaining cases
\begin{enumeratei-continued}
\item  Inequalities \eqref{inv:1} and \eqref{inv:1'}   will hold in any $G'$, provided that the right hand side does not increase.
\end{enumeratei-continued}

\smallskip
\case{9}{There is $C\in\mathcal{C}$ with $C\cap U\neq\emptyset$.}

\smallskip
As Case 7 does not apply,  $|C\cap U|=1$.   We put $I=C\cap U$. Say that $C\setminus U= \set{u,v}$ then we put $\set{u,v}$ into class $\mathcal{B'}$. It is straightforward  that vertices from $\{v,u\}$ satisfy \eqref{inv:2.1}. They also satisfy \eqref{inv:2.2} as
\[
f'(u)+f'(v)=f(u)+f(v)\geq \norm{V(G)}-1=\norm{V(G')}.
\]

\smallskip

\case{10}{There is $B\in\mathcal{B}$ with $B\cap U\neq\emptyset$.}

\smallskip

We put $I=B\cap U$. Say that $B \setminus U= \set{u}$. We put $\set{u}$ to the very beginning of the class $\mathcal{A'}$.
By the observations above, all the inequalities hold, and hence $G'$ is on-line $f'$-choosable (note that for Inequalities \eqref{inv:1} and \eqref{inv:1'},
the right hand side does not increase, as $k'_2$ decreases by $1$ and the index increases by $1$).

It is easy to see that one of the 10 cases above occurs and hence $G$ is on-line $f$-choosable.
\end{proof}

The following theorem is an immediate consequence of Lemma \ref{k3k:induction}.
\begin{theorem}
\label{thm:multi3}
$\olch(K_{3\star k})\leq\frac{3}{2}k$, for any positive integer $k$.
\end{theorem}

In the next theorem we prove that on-line Ohba Conjecture is true for graphs with independence number at most 3.

\begin{theorem}\label{thm:off-line-independence-3}
\label{thm:ohba3}
	If $G$ is a graph with independence number at most 3 and $|V(G)|\leq 2\chi(G)$, then $\chi(G)=\olch(G)$.
\end{theorem}
\begin{proof}[Proof of Theorem \ref{thm:off-line-independence-3}]
Without loss of generality, we can assume that $G$ is a complete multipartite graph with parts of size at most 3.  We are going to verify that $G$ satisfies Lemma \ref{k3k:induction} with $\mathcal{S}= \emptyset$,  $f\equiv \chi(G)$ and arbitrary order on the class $\mathcal{A}$ (when $\mathcal{S}=\emptyset$ the remaining classes of the partition are determined). Let $k_1, k_2, k_3$ denote the sizes  of parts of sizes 1,2,3, respectively.

Inequalities for the single vertices \eqref{inv:1}, \eqref{inv:2.1}, \eqref{inv:3.1} hold as $f(v)=\chi(G)= k_1+k_2+k_3$. Condition on pairs of vertices \eqref{inv:2.2}, \eqref{inv:3.2}
hold since $f(u)+f(v)= 2\chi(G) \geq |V(G)|$ (by the assumption on $G$).
Moreover adding $\chi(G)= k_3+k_2+k_1$ on both sides of the inequality \eqref{inv:3.2} gives \eqref{inv:3.3}.

Now, by Lemma \ref{k3k:induction} $G$ is on-line chromatic-choosable.
\end{proof}

\section{Off-line considerations}
The main obstacle for translating results from off-line to on-line case concerns the application of Hall Theorem, which can not be applied directly to the on-line version. In the following section we present how some previously known results can be derived by the methods we used for Theorem \ref{thm:ohba3} and \ref{thm:multi3}.
The application of Hall Theorem is encapsulated in the following lemma used in many proofs concerning choosability (see e.g.\ \cite{Kierstead00,RS05}).
\begin{lemma}
\label{lem:smallNbOfCol}
A graph $G$ is $k$-choosable if it is $L$-colourable for every $k$-list assignment $L$ such that $|\bigcup_{v\in V} L(v) | < |V|$.
\end{lemma}

\subsection{Chromatic choosable graphs with independence number 3}

Ohba Conjecture was proved to be true for graphs with independence number at most $3$ already in \cite{SHZL09}. 
We present an alternative proof based on Lemma \ref{k3k:induction}. 
The interesting point here is that we use Lemma \ref{lem:smallNbOfCol} only to prove that it is possible to color two vertices with one colour in such a way that the remaining graph is on-line choosable.

\begin{theorem}\label{thm:on-line-independence-3}
	If $G$ is a graph with independence number at most $3$ and $|V(G)|\leq 2\chi(G)+1$, then $\chi(G)=\ch(G)$.
\end{theorem}
\begin{proof}
	For a contradiction let $G$ be a counterexample with minimum number of vertices.
Let $L$ be a $\chi(G)$-list assignment such that $G$ is not $L$-colourable.
By Theorem \ref{thm:off-line-independence-3}, we may assume that $|V(G)|=2\chi(G)+1$ and by Lemma \ref{lem:smallNbOfCol} we assume that the number of colours occurring on all the list is at most $2 \chi(G)$.
	
We can also assume that for every part $\{u,v\}$ of size $2$ the lists $L(u)$ and $L(v)$ are disjoint.
If not, then we pick a colour $c\in L(u)\cap L(v)$ and use it to colour both vertices. 
The remaining graph $G'=G-\set{u,v}$ still satisfies $|V(G')|\leq 2\chi(G')+1$. By the minimality of $G$,
 we know that $G'$ is chromatic-choosable. Hence $G'$ is colourable from $L-\set{c}$,  implying that $G$ is colourable from $L$,
 a contradiction. For the   same reason there is no colour that belongs to all three lists of vertices of any part of size 3 in $G$.
	
	As $|V(G)|=  2\chi(G)+1$ there  exists at least one part of size $3$ in $G$, say $\{u,v,w\}$.
Each vertex has a list of size $\chi(G)$ and the total number of colours is at most $2\chi(G)$,
therefore there  exists a colour $c$ which belongs to lists of two vertices from this part,
say $c\in L(u) \cap L(w)$.

We are going to construct an $L$-colouring of $G$ in two steps. First, we use $c$ to colour $u$ and $w$, remove them from $G$ and remove
colour $c$ from all lists. Than we prove that the remaining graph $G'=G-\set{u,w}$ is on-line $f'$-choosable, where
\[
	f'(v)= \begin{cases}
			\chi(G)   &\mbox{if } c \notin L(v),\\
			\chi(G)-1 & \mbox{if } c \in L(v).
		\end{cases}
\]
In particular, $G'$ can be coloured from $L-\set{c}$, which finishes the colouring of $G$ and gives the final contradiction.

The only thing we need to verify is that $G'$ and $f'$ satisfy the assumptions of Lemma \ref{k3k:induction} with $\mathcal{S}=\emptyset$ and parts from $\mathcal{A}$ ordered in such a way that the part $\set{v}$ has the greatest index. Let $k_1, k_2, k_3,k_1',k_2',k_3'$ denote the numbers of parts of size $1$, $2$ and $3$ in $G$ and $G'$, respectively. We have
\[
k'_1 =k_1+1,\qquad k'_2=k_2,\qquad k'_3 =k_3 -1.
\]
Inequalities \eqref{inv:2.1}, \eqref{inv:3.1} hold as for any $x$ in part of size $2$ or $3$ in $G'$
\[
f'(x) \geq \chi(G)-1 \geq k_3+k_2-1=k_3'+k_2'
\]
The part of size $1$, say $\set{x}$, with index less than $k'_1$ satisfies \eqref{inv:1} as
\[
f'(x) \geq \chi(G)-1=k_3'+k_2'+k_1'-1.
\]
The remaining part of size $1$, namely $\{v\}$, satisfies \eqref{inv:1} as $f(v)=\chi(G)=\chi(G')$ (as $c\not\in L(v)$). Inequalities \eqref{inv:2.2} hold since colour $c$ belongs to the list of at most one vertex in every part of size $2$ in $G'$. Therefore, for any $\set{x,y}$ part of size $2$ in $G'$ we have
\[
f'(x)+f'(y)\geq 2\chi(G)-1=\norm{V(G')}-1.
\]
It remains to verify inequalities \eqref{inv:3.2} and \eqref{inv:3.3}. Let $x$, $y$, $z$ be any three vertices forming a part of size $3$ in $G'$. Then
\begin{align*}
f'(x)+f'(y)&\geq 2\chi(G)-2=\norm{V(G')}-1,\\
f'(x)+f'(y)+f'(z)&\geq 3\chi(G)-2=\norm{V(G')}-1+k_3'+k_2'+k_1'.
\end{align*}
The latter inequality follows from the fact $c$ is not in all three $L(x)$, $L(y)$, $L(z)$.
\end{proof}

\subsection{The complete multipartite graph $K_{3\star k}$}

\noindent There are not so many graphs for which the exact value of a choice number is known.
In \cite{Kierstead00}, Kierstead proved that $\ch(K_{3\star k}) = \lceil (4k-1)/3 \rceil$.
We present an alternative proof of this result once again using Hall Theorem encapsulated in Lemma \ref{lem:smallNbOfCol}.

\begin{theorem}[Kierstead 2000]\label{thmk}
For any positive integer $k$, $\chnol(K_{3\star k}) = \lceil \frac{4k-1}{3} \rceil$.
\end{theorem}

The  lower bound $\chnol(K_{3\star k}) \ge  \lceil \frac{4k-1}{3} \rceil$ was given by Erd\"{o}s, Rubin and Taylor \cite{ERT79}.
As the proof is very short, we include it here for the convenience of the reader.   Let $q=\lceil \frac{4k-1}{3} \rceil-1$.
Let $A, B, C$ be disjoint colour sets with $|A|=\lfloor q/2 \rfloor$ and $|B|=|C|=\lceil q/2 \rceil$. Assume the parts of $K_{3\star k}$
are $V_i=\{x_i,y_i,z_i\}$ for $i=1,2,\ldots, k$. Let $L(x_i) = A \cup B, L(y_i)=B \cup C$ and $L(z_i)=A \cup C$. Then
$|L(v)| \ge q$ for each vertex $v$, and if $f$ is an $L$-colouring of $K_{3\star k}$, then $f$ uses at least $2$ colours on $V_i$, and hence
the total number of used colours is at least $2k$. However, straightforward calculation shows that $|A \cup B \cup C| \le 2k-1$. Therefore
$K_{3\star k}$ is not $L$-colourable and hence $\chnol(K_{3\star k}) \ge q+1 = \lceil \frac{4k-1}{3} \rceil$.

The inequality  $\chnol(K_{3\star k}) \le \lceil \frac{4k-1}{3} \rceil$ is a straightforward consequence of the following lemma.

\begin{lemma}
Let $G$ be a complete multipartite graph with parts of size $1$ and $3$. Let  $\mathcal{A},$ $\mathcal{S}$, $\mathcal{C}$ be a partition of the set of parts of $G$ into classes  such that $\mathcal{A}$ and $\mathcal{S}$ contains only parts of size $1$, while $\mathcal{C}$ contains all parts of size $3$. Let $k_1, s, k_3$ denote the cardinalities of classes $\mathcal{A}$, $\mathcal{S}$, $\mathcal{C}$, respectively. Suppose that class $\mathcal{A}$ and $\mathcal{S}$ are ordered, i.e.\ $\mathcal{A}= (A_1, \ldots, A_{k_1})$
and $\mathcal{S}=(S_1, \ldots, S_s)$. If $f:V(G)\to \NN$ is a function for which the following conditions hold
\begin{alignat}{2}
f(v)&\geq k_3+i, &&\quad\text{for all $1\leq i \leq k_1$ and $v\in A_i$}\tag{1}\label{invK:1}\\
f(v)&\geq 2 k_3+k_1+i, &&\quad\text{for all $1\leq i \leq s$ and $v\in S_i$}\tag{1'}\label{invK:1*}\\
f(v)&\geq k_3, &&\quad\text{for all $v\in C\in\mathcal{C}$}\tag{3.1}\label{invK:3.1}\\
f(u)+f(v)&\geq 2 k_3 + k_1, &&\quad\text{for all $u,v\in C\in\mathcal{C}$}\tag{3.2}\label{invK:3.2}\\
\sum_{v\in C}f(v)&\geq 4 k_3+2k_1+s-1, &&\quad\text{for all $C\in\mathcal{C}$}, \tag{3.3}\label{invK:3.3}
\end{alignat}
then $G$ is $f$-choosable.	
\end{lemma}

\begin{proof}
Assume the lemma is not true. Let $G$ be a multipartite graph with parts divided into $\mathcal{A}$, $\mathcal{S}$, $\mathcal{C}$, and let $f$ be a function fulfilling the inequalities \eqref{invK:1}-\eqref{invK:3.3}  while $G$ is not $f$-choosable. Moreover, suppose $G$ is a counterexample with the minimum possible number of vertices. By Lemma \ref{lem:smallNbOfCol} there exists a list assignment $\set{L(v)}_{v\in V(G)}$ with each $\norm{L(v)}=f(v)$ and $|\bigcup_{v\in V(G)} L(v) | \leq |V(G)|-1 = 3k_3+ k_1 +s-1$ such that $G$ is not $L$-colourable.

The claims below prove a series of properties of $G$ and list assignment $L$. In the arguments we often make use the minimality of $G$ and consider some smaller graphs with modified list assignment. The modified graph will be denoted by $G'$ and, unless otherwise stated, the classes of its vertices  $\mathcal{A'}$, $\mathcal{S'}$ and $\mathcal{C'}$, together with  orders on $\mathcal{A'}$ and $\mathcal{S'}$, are inherited from $G$. The parameters $k_1',s',k_3'$ correspond to the analogous parameters of $G'$. The modified list assignment is  going to be denoted by $L'(v)$ and $f'(v)=\norm{L'(v)}$ for all $v\in V(G')$.
	
\smallskip

\claim{0}{For any $C\in\mathcal{C}$ we have $\bigcap_{v\in C}L(v)=\emptyset$.}
\begin{proof}
Suppose there is $C\in\mathcal{C}$ with $c\in\bigcap_{v\in C}L(v)$. We colour all vertices of $C$ with $c$ and consider the smaller graph $G'=G-C$ with list assignment $L'(v)=L(v)-\set{c}$. It is easy to verify that $G'$ (with $\mathcal{A'}$, $\mathcal{S'}$, $\mathcal{C'}$ inherited from $G$) and $f'$ satisfies the assumptions of the lemma. By the minimality of $G$,   $G'$ is $L'$-colourable. This implies that   $G$ is $L$-colourable, in contrary to our assumption.
\end{proof}

\claim{1}{For any $u,v\in C\in\mathcal{C}$ if $f(u)+f(v)= 2 k_3 +k_1$, then $L(u)\cap L(v)= \emptyset$.}
\begin{proof}
Suppose that for some part $C= \{u,v,w\}$ we have $f(u)+f(v)= 2 k_3 +k_1$ and there exist $c \in L(u)\cap L(v)$.
Then we colour $u$ and $v$ with $c$, and consider the smaller graph $G'=G-\set{u,v}$ with lists $L'(x)=L(v)-\set{c}$ for all $x\in V(G')$. The partition $\mathcal{A'}$, $\mathcal{C'}$ is inherited from $G$ and $\mathcal{S'}=(\set{w},S_1,\ldots,S_{s})$ has one more part, namely $\set{w}$, while all other parts have shifted index, i.e., $S_{i+1}'=S_{i}$ for $1\leq i\leq s$. In particular, $k_1'=k_1$, $s'=s+1$, $k_3'=k_3-1$. Note that the inequality \eqref{invK:1*} holds for $S'_{1}=\set{w}$ as
\[
f'(w)=f(w)\geq (4k_3+2k_1+s-1) - (2k_3+k_1) = 2k_3+k_1+s-1 = 2k_3'+k_1'+1,
\]
and \eqref{invK:1*} holds for $S'_{i+1}=S_{i}=\set{x}$ for $1\leq i\leq s$ as
\[
f'(x)\geq f(x)-1\geq (2k_3+k_1+i)-1=2k_3'+k_1'+i+1.
\]
Again, it is easy to verify that $G'$ with $f'$ satisfies the assumptions of the lemma. Hence  $G'$ is $L'$-colourable, implying that $G$ is $L$-colourable, a contradiction.
\end{proof}	
\claim{2}{For any $v\in C\in\mathcal{C}$ we have $f(v)>k_3$, i.e., the inequality \eqref{invK:3.1} is not tight.}
\begin{proof}
In order to get a contradiction suppose that $\set{v,u,w}=C\in\mathcal{C}$ and $f(v)=k_3$. We separate the argument into two cases:
\begin{itemize}
\item $L(v) \cap (L(u)  \cup L(w))\neq \emptyset$. Without loss of generality, assume that $L(v) \cap L(u) \neq \emptyset$.
Let $c\in L(v) \cap L(u)$. We colour $u$ and $v$ with $c$, and consider the smaller graph $G'=G-\set{u,v}$ with lists $L'(x)=L(x)-\set{c}$ for all $x\in V(G')$. The partition $\mathcal{S'}$, $\mathcal{C'}$ is inherited from $G$ and $\mathcal{A'}=(A_1,\ldots,A_{k_1},\set{w})$ has one more part, namely $\set{w}$, appended to the inherited ordering. In particular, $k_1'=k_1+1$, $s'=s$, $k_3'=k_3-1$. Note that the inequality \eqref{invK:1} holds for $A'_{k_1'}=\set{w}$ as
\[
f'(w)=f(w)>(2k_3+k_1)-k_3=k_3'+ k_1'.
\]
Let $x,y\in C\in\mathcal{C'}$. Inequality \eqref{invK:3.2} for $x$ and $y$ hold as either $f(x)+f(y)>2k_3+k_1$ and therefore
\[
f'(x)+f'(y)\geq f(x)+f(y)-2 > 2k_3+k_1-2 = 2k_3'+k_1'-1,
\]
or $f(x)+f(y)=2k_3+k_1$ and therefore by Claim 1 $L(x)$ and $L(y)$ are disjoint.
\[
f'(x)+f'(y)\geq f(x)+f(y)-1 =2k_3+k_1-1=2k_3'+k_1'.
\]
With these observations, it is easy to verify that $G'$ with $f'$ satisfies the assumptions of the lemma. Hence $G'$ is $L'$-colourable and therefore $G$ would be $L$-colourable, a contradiction.
\item $L(v) \cap (L(u) \cup L(w))= \emptyset$. Then by \eqref{invK:3.3} and our assumption $f(v)=k_3$ we get that
\[
f(u)+f(w)\geq (4k_3+2k_1+s-1)-k_3=3k_3+2k_1+s-1.
\]
On the other hand the total number of colours is at most $3 k_3 + k_1 +s-1$ and as $L(v)$ is disjoint with $L(u)\cup L(w)$ we get $\norm{L(u) \cup L(w)} \leq 2 k_3 +k_1+s -1$. Combining the two  inequalities above we obtain
\[
\norm{L(u) \cap L(w)} \geq k_3 + k_1.
\]
We colour vertex  $v$ by any colour $c\in L(v)$. Then we consider graph $G'=G-\set{v,u,w}+\set{x}$, where $x$ is a brand new vertex which is convenient to be seen as a merger of $u$ and $w$. Let $L'(y)=L(y)-\set{c}$ for all $y\in V(G')-\set{x}$ and $L'(x)=L(u)\cap L(w)$. The partition $\mathcal{S'}$, $\mathcal{C'}$ is inherited from $G$ and $\mathcal{A'}=(A_1,\ldots,A_{k_1},\set{x})$ has one more part, namely $\set{x}$, appended to the inherited ordering. In particular, $k_1'=k_1+1$, $s'=s$, $k_3'=k_3-1$. Note that the inequality \eqref{invK:1} holds for $A'_{k_1'}=\set{x}$ as
\[
f'(x)=\norm{L(u) \cap L(w)} \geq k_3 + k_1=k_3'+k_1'.
\]
The other inequalities for $G'$ and $f'$ hold for the same reasons as before.  So $G'$ is $L'$-colourable.
We obtain an $L$-colouing of $G$, by colouring the vertices $u$ and $w$ with the colour of $x$ and colouring $v$ with $c$, a contradiction.
	\end{itemize}
\end{proof}
\claim{3} {$k_1=0$.}
\begin{proof}
Suppose that $k_1\neq 0$. Then let $A_1=\set{v}$. We colour $v$ with any colour $c\in L(v)$ and consider the smaller graph $G'=G-\set{v}$ with lists $L'(x)=L(x)-\set{c}$ for all $x\in V(G')$. The partition $\mathcal{A'}=(A_2,\ldots,A_{k_1})$, $\mathcal{S'}$, $\mathcal{C'}$ is inherited from $G$. Note that $\mathcal{A'}$ has one part less and $k_1'=k_1-1$, $s'=s$, $k_3'=k_3$. Now, we verify the inequalities \eqref{invK:1}-\eqref{invK:3.3} for $G'$ and $f'$:
\begin{itemize}
	\item \eqref{invK:1} holds as the indices of parts are decreased, i.e. $A_i'=A_{i+1}$ for $1\leq i < k_1'$;
	\item \eqref{invK:1*} holds as $k_1$ decreases,
	\item \eqref{invK:3.1} holds as, by Claim 2, it is not tight in $G$,
	\item \eqref{invK:3.2} holds for $x,y\in C\in\mathcal{C'}$ as $k_1$ decreases and either \eqref{invK:3.2} is not tight for $u$, $v$ in $G$, or $f'(x)+f'(y)\geq f(x)+f(y)-1$ (by Claim 1);
	\item \eqref{invK:3.3} holds as $k_1$ decreases by 1 and the left hand side decreases by at most $2$ (by Claim 0).
\end{itemize}
Once again by minimality of $G$ we get that $G'$ is $f'$-choosable, and that gives that $G$ is $L$-colourable, a contradiction.
\end{proof}

We are now ready to derive the final contradiction. If $k_3=0$ then $G$ has only parts of size $1$ in $\mathcal{S}$ and it is immediate that $G$ is $f$-choosable. Assume $k_3 \ne 0$. Recall that the total number of colors in all lists is at most $3k_3+s-1$. Let $\{u,v,w\}$ be a part of size $3$.
Then $f(u)+f(v)+f(w) \ge 4k_3+s-1>3k_3+s-1$ and therefore there must be a colour $c$ which appears in two out of three colour sets $L(u)$, $L(v)$, $L(w)$, say $c\in L(u)\cap L(v)$.

We colour $u$ and $v$ with $c$ and consider $G'=G-\set{u,v}$ with lists $L'(x)=L(x)-\set{c}$. Again, the partition $\mathcal{S'}$, $\mathcal{C'}$ is inherited from $G$ and we simply put $\mathcal{A'}=(\set{w})$. Thus, $k_1'=1$, $s'=s$, $k_3'=k_3-1$. We verify the inequalities \eqref{invK:1}-\eqref{invK:3.3} for $G'$ with $f'$. The inequality \eqref{invK:1} for $A_1'=\set{w}$ holds as
\[
f'(w)=f(w)>k_3=k_3'+1.
\]
All the other inequalities hold for analogous reasons as before. Once again, by minimality of $G$, we get that $G'$ is $f'$-choosable, and that gives that $G$ is $L$-colourable, a contradiction.
\end{proof}


\bibliographystyle{siam}
\bibliography{ohba}

\begin{thebibliography}{10}

\bibitem{CZ2011}
{\sc T.~Chang and X.~Zhu}, {\em On-line $3$-choosable planar graphs}, Taiwanese
  Journal of Mathematics,  (to appear).

\bibitem{ERT79}
{\sc P.~Erd{\H{o}}s, A.~L. Rubin, and H.~Taylor}, {\em Choosability in graphs},
  in Proceedings of the {W}est {C}oast {C}onference on {C}ombinatorics, {G}raph
  {T}heory and {C}omputing ({H}umboldt {S}tate {U}niv., {A}rcata, {C}alif.,
  1979), Congress. Numer., XXVI, Winnipeg, Man., 1980, Utilitas Math.,
  pp.~125--157.

\bibitem{GM1997}
{\sc S.~Gravier and F.~Maffray}, {\em Choice number of {$3$}-colorable
  elementary graphs}, Discrete Math., 165/166 (1997), pp.~353--358.

\bibitem{Gutowski2011}
{\sc G.~Gutowski}, {\em Mr. paint and mrs. corrector go fractional}, Electron.
  J. Combin., 18 (2011), p.~Research Paper 140.

\bibitem{HC1992}
{\sc R.~H{\"a}ggkvist and A.~Chetwynd}, {\em Some upper bounds on the total and
  list chromatic numbers of multigraphs}, J. Graph Theory, 16 (1992),
  pp.~503--516.

\bibitem{HWZ2010}
{\sc P.~Huang, T.~Wong, and X.~Zhu}, {\em Application of polynomial method to
  on-line colouring of graphs}, European J. Combin.,  (to appear).

\bibitem{JensenToft}
{\sc T.~R. Jensen and B.~Toft}, {\em Graph coloring problems},
  Wiley-Interscience Series in Discrete Mathematics and Optimization, John
  Wiley \& Sons Inc., New York, 1995.
\newblock A Wiley-Interscience Publication.

\bibitem{Kierstead00}
{\sc H.~A. Kierstead}, {\em On the choosability of complete multipartite graphs
  with part size three}, Discrete Math., 211 (2000), pp.~255--259.

\bibitem{KKLZ2011}
{\sc S.~Kim, Y.~Kwon, D.~Liu, and X.~Zhu}, {\em On-line list colouring of
  complete multipartite graphs}, Manuscript,  (2011).

\bibitem{KW2001}
{\sc A.~V. Kostochka and D.~R. Woodall}, {\em Choosability conjectures and
  multicircuits}, Discrete Math., 240 (2001), pp.~123--143.

\bibitem{NRW12}
{\sc J.~Noel, B.~Reed, and H.~Wu}, {\em A proof of a conjecture of ohba},
  Manuscript,  (2012).

\bibitem{OH2002}
{\sc K.~Ohba}, {\em On chromatic-choosable graphs}, J. Graph Theory, 40 (2002),
  pp.~130--135.

\bibitem{RS05}
{\sc B.~Reed and B.~Sudakov}, {\em List colouring when the chromatic number is
  close to the order of the graph}, Combinatorica, 25 (2005), pp.~117--123.

\bibitem{Schauz}
{\sc U.~Schauz}, {\em Mr. {P}aint and {M}rs. {C}orrect}, Electron. J. Combin.,
  16 (2009), pp.~Research Paper 77, 18.

\bibitem{Schauz4}
\leavevmode\vrule height 2pt depth -1.6pt width 23pt, {\em Paintability version
  of the combinatorial nullstellensatz, and list colorings of $k$-partite
  $k$-uniform hypergraphs}, Electron. J. Combin.,  (submitted).

\bibitem{SHZL09}
{\sc Y.~Shen, W.~He, G.~Zheng, and Y.~Li}, {\em Ohba's conjecture is true for
  graphs with independence number at most three}, Appl. Math. Lett., 22 (2009),
  pp.~938--942.

\bibitem{tuzasurvey}
{\sc Z.~Tuza}, {\em Graph colorings with local constraints---a survey},
  Discuss. Math. Graph Theory, 17 (1997), pp.~161--228.

\bibitem{Vizing76}
{\sc V.~G. Vizing}, {\em Coloring the vertices of a graph in prescribed
  colors}, Diskret. Analiz,  (1976), pp.~3--10, 101.

\bibitem{Zhuonline}
{\sc X.~Zhu}, {\em On-line list colouring of graphs}, Electron. J. Combin., 16
  (2009), pp.~Research Paper 127, 16.

\end{thebibliography}
\end{document}